\documentclass[10pt,reqno]{amsart}
\usepackage{graphicx}
\usepackage{amsfonts}
\usepackage{amssymb}
\usepackage{amsmath}
\usepackage{amsxtra}
\usepackage{latexsym}
\usepackage{epstopdf}
\usepackage{mathrsfs}
\usepackage{esint}
\usepackage{xcolor}
\usepackage{enumitem}

\newtheorem{theorem}{Theorem}
\newtheorem{lemma}[theorem]{Lemma}

\newtheorem{Rule}{Rule}

\theoremstyle{definition}

\newtheorem{Assumption}{Assumption}

\theoremstyle{remark}

\begin{document}

\title[On saturation of the discrepancy principle]
{On saturation of the discrepancy principle for nonlinear Tikhonov regularization in Hilbert spaces}

\author{Qinian Jin}

\address{Mathematical Sciences Institute, Australian National
University, Canberra, ACT 2601, Australia}
\email{qinian.jin@anu.edu.au} \curraddr{}

\keywords{Nonlinear ill-posed problems, Tikhonov regularization, the discrepancy principle, saturation}

\begin{abstract}
In this paper we revisit the discrepancy principle for Tikhonov regularization of nonlinear 
ill-posed problems in Hilbert spaces and provide some new and improved saturation results 
under less restrictive conditions, comparing with the existing results in the literature. 
\end{abstract}

\def\d{\delta}
\def\l{\langle}
\def\r{\rangle}
\def\a{\alpha}
\def\la{\lambda}
\def\N{\mathcal N}

\maketitle

\section{\bf Introduction}

Let $X$ and $Y$ be two  Hilbert spaces whose inner products and norms are denoted by the same notations
$\l\cdot, \cdot\r$ and $\|\cdot\|$. Let $F: \mbox{dom}(F)\subset X \to Y$ be a nonlinear 
operator with domain $\mbox{dom}(F)$. We consider the equation
\begin{equation}\label{TRH1}
F(x)=y
\end{equation}
and assume (\ref{TRH1}) is solvable in $\mbox{dom}(F)$, i.e. $y \in \mbox{Ran}(F)$, the range of $F$.
Since the solution of (\ref{TRH1}) may not be unique, it is necessary to select a right one.
To this end, we make a first guess $x^*\in \mbox{dom}(F)$ and try to find a solution closest to $x^*$. 
We say $x^\dag\in \mbox{dom}(F)$ is a $x^*$-minimum-norm solution of (\ref{TRH1}) if
$
F(x^\dag)=y
$
and
$$
\|x^\dag-x^*\|=\min_{x\in \mbox{dom}(F)} \{\|x-x^*\|: F(x)=y\}.
$$

We consider the situation that (\ref{TRH1}) is ill-posed in the sense that its solution does not depend 
continuously on the data. The exact data $y$ is usually unavailable in practical applications. Instead we 
only have a noisy data $y^\d$ obtained by measurement which satisfies
\begin{equation}\label{TRH2}
\|y^\d-y\|\le \d
\end{equation}
with a known noise level $\d>0$. One of the important questions is then to reconstruct $x^\dag$ by using 
$y^\d$. This requires the regularization techniques. The most famous method is Tikhonov regularization in 
which the solution of the minimization problem
\begin{equation}\label{TRH3}
\min_{x\in \mbox{dom}(F)}\left\{\|F(x)-y^\d\|^2+\a\|x-x^*\|^2\right\}
\end{equation}
is used to approximate $x^\dag$, where $\a>0$ is the so-called regularization parameter. The 
existence of $x^*$-minimum-norm solution and the regularized solution $x_\a^\d$ can be guaranteed 
when $F$ is weakly closed; see \cite{EKN89}. The solution 
of (\ref{TRH3}) may not be unique in general, we will use $x_\a^\d$ to denote any solution of (\ref{TRH3}).
When $x_\a^\d$ is used to approximate $x^\dag$, the accuracy depends crucially on the choice of the 
regularization parameter $\a$ which can be selected by either {\it a priori} or {\it a posteriori} 
rules (\cite{EKN89,JH1999,SEK1993,TJ2003}). One of the most prominent {\it a posteriori} parameter 
choice rule is Morozov's discrepancy principle which can be stated as follows. 

\begin{Rule}[The discrepancy principle]\label{Rule:DP}
Let $\tau\ge 1$ be a given number. Define $\a(\d, y^\d) > 0$ to be the number such that 
$\|F(x_{\a(\d, y^\d)}^\d) - y^\d \| = \tau \d$. 
\end{Rule}

One may refer to \cite{Jin1999} for a discussion on the existence of $\a(\d, y^\d)$ satisfying 
Rule \ref{Rule:DP}. With $\a(\d, y^\d)$ chosen by Rule \ref{Rule:DP}, it has been shown in \cite{EKN89} 
that $x_{\a(\d, y^\d)}^\d \to x^\dag$ as $\d\to 0$ and if $F$ is Fr\'{e}chet differentiable in a 
neighborhood of $x^\dag$ whose Fr\'{e}chet derivative $F'(x)$ is Lipschitz continuous with constant 
$L$ and if the source condition $x^\dag - x^* = F'(x^\dag)^* u$ holds with $L\|u\|<1$, then 
\begin{align}\label{rate}
\sup\left\{\|x_{\a(\d, y^\d)}^\d - x^\dag\|: \|y^\d - y\| \le \d\right\} = O(\d^{1/2})
\end{align}
for the worst case error. It is natural to ask if the rate in (\ref{rate}) can be improved under a higher 
order source condition $x^\dag - x^* = \mbox{Ran}((F'(x^\dag)^* F'(x^\dag))^\nu)$ with $\nu >1/2$. For 
compact linear operator $F$, Groetsch showed in \cite{G1983,Gr84} that if the rate $O(\d^{1/2})$ in 
(\ref{rate}) is replaced by $o(\d^{1/2})$ then it must have $x^*= x^\dag$. Thus, for a general choice 
of the initial guess $x^*$, the rate in (\ref{rate}) is the best possible in the worst case scenario which 
is known as the saturation of the discrepancy principle for linear Tikhonov regularization in Hilbert spaces.  
In \cite{S1993a} Scherzer gave a saturation result on Rule \ref{Rule:DP} for nonlinear Tikhonov regularization 
in Hilbert spaces. His proof relies on several conditions on the nonlinear operator $F$ some of which seem 
restrictive. Furthermore, the argument in \cite{S1993a} requires $\tau$ to be sufficiently large which seems 
unnatural. In this note we will provide new and improved saturation results for Rule \ref{Rule:DP} 
for nonlinear Tikhonov regularization (\ref{TRH3}) in Hilbert spaces under less restrictive conditions.

\section{\bf Main results}

The saturation result proved in \cite{S1993a} on Rule \ref{Rule:DP} for nonlinear Tikhonov regularization in 
Hilbert spaces is based on the following technical conditions on the operator $F$, where $B_\rho(x^\dag) : = 
\{x\in X: \|x- x^\dag\| < \rho\}$. 

\begin{Assumption}\label{Ass.1}
{\it There is $\rho>0$ such that $B_\rho(x^\dag) \subset \emph{dom}(F)$ and $F$ is Fr\'{e}chet dfferentiable 
on $B_\rho(x^\dag)$. Moreover, there exists $\kappa_0\ge 0$ such that for any $x, z \in B_\rho(x^\dag)$ 
there is a bounded linear operator $R(x,z): X \to X$ such that 
$$
F'(x) = F'(z) R(x, z)  \quad \mbox{and} \quad \|I - R(x, z)\| \le \kappa_0 \|x-z\|.
$$   
}
\end{Assumption}

\begin{Assumption}\label{Ass.2}
{\it There exists $\kappa_1\ge 0$ such that for any $x, z \in B_\rho(x^\dag)$ there is a bounded linear 
operator $S(x,z): Y \to Y$ such that 
$$
F'(x) = S(x,z) F'(z)   \quad \mbox{and} \quad \|I - S(x, z)\| \le \kappa_1 \|x-z\|.
$$
}
\end{Assumption}

\begin{theorem}[\cite{S1993a}]\label{thm:os}
Let Assumption \ref{Ass.1} and Assumption \ref{Ass.2} hold. Assume that $F'(x^\dag)$ 
is compact with infinite rank and $\max\{\kappa_0, \kappa_1\} \|x^*-x^\dag\|$ is 
sufficiently small. Consider Tikhonov regularization (\ref{TRH3}) and let 
$\a(\d, y^\d)$ be determined by Rule \ref{Rule:DP} with sufficiently large 
$\tau>1$. If 
$$
\sup\left\{ \|x_{\a(\d, y^\d)}^\d - x^\dag\|: \|y^\d-y\|\le \d\right\} = o(\d^{1/2})
$$
as $\d\to 0$, then $x^\dag = x^*$. 
\end{theorem}

The proof idea of Theorem \ref{thm:os} given in \cite{S1993a} is to connect Rule \ref{Rule:DP}  
for the nonlinear Tikhonov regularization (\ref{TRH3}) with a discrepancy principle for its 
linearization at the sought solution $x^\dag$ and then use the saturation result of Groetsch 
(\cite{G1983,Gr84}) for linear Tikhonov regularization. More precisely, replacing $F(x)$ in 
(\ref{TRH3}) by its linearization $y + A (x-x^\dag)$ at $x^\dag$, where $A := F'(x^\dag)$, one 
may consider the linearized minimization problem
\begin{align}\label{LTRH}
\min_{x\in X}\left\{\|A x-q^\d\|^2
+\a\|x-x^*\|^2\right\},\quad q^\d := A x^\dag+y^\d-y
\end{align}
whose unique minimizer is given by
\begin{equation}\label{LTRH1}
\hat{x}_\a^\d=x^\dag+(\a I+A^*A)^{-1} \left(\a (x^*-x^\dag)+A^*(y^\d-y)\right).
\end{equation}
It is then shown that if $\a$ is chosen by Rule \ref{Rule:DP}, there must exists 
a constant $c>1$ such that 
\begin{align}\label{LDP}
\|A \hat x_\a^\d - q^\d \| = \hat \tau \d
\end{align}
for some number $\hat \tau \in [1, c]$. Once this is achieved, the saturation result 
on the discrepancy principle for linear Tikhonov regularization can be used directly 
to conclude $x^* - x^\dag =0$. In order to achieve the above connection between 
Rule \ref{Rule:DP} for nonlinear Tikhonov regularization and a discrepancy 
principle for its linearization,one needs to compare $\|F(x_\a^\d) - y^\d\|$ with 
$\|A \hat x_\a^\d - q^\d\|$ and guarantee the existence of a constant $c$ such that 
(\ref{LDP}) holds for some $\hat \tau \in [1, c]$. These require the number $\tau$ 
in Rule \ref{Rule:DP} to be sufficiently large and the nonlinear operator $F$ to satisfy 
Assumption \ref{Ass.1} and Assumption \ref{Ass.2} with sufficiently small 
$\max\{\kappa_0, \kappa_1\} \|x^* - x^\dag\|$. 

In the following we will present some new and improved saturation results on Rule \ref{Rule:DP} 
for nonlinear Tikhonov regularization (\ref{TRH3}) in Hilbert spaces under less restrictive 
conditions. Our first result shows that the conclusion of Theorem \ref{thm:os} holds for 
Rule \ref{Rule:DP} with any $\tau >0$ under merely Assumption \ref{Ass.1} and the smallness 
of $\max\{\kappa_0, \kappa_1\} \|x^* - x^\dag\|$ is not required.

\begin{theorem}\label{TRH.thm1}
Let Assumption \ref{Ass.1} hold. Assume that, for the operator $A:= F'(x^\dag)$, there 
exists a sequence of positive numbers $\{\la_k\}$ in the spectrum $\sigma(A A^*)$ of $AA^*$ 
such that $\lim_{k\to \infty} \la_k =0$. Consider the nonlinear Tikhonov regularization 
(\ref{TRH3}) and assume that there is a constant $\tau>0$ such that, for any small $\d>0$ 
and any $y^\d$ satisfying $\|y^\d - y\| \le \d$, a number $\a(\d, y^\d) >0$ can be chosen such 
that $\|F(x_{\a(\d, y^\d)}^\d) - y^\d\| \le \tau \d$. If 
\begin{align}\label{8}
\sup\left\{ \|x_{\a(\d, y^\d)}^\d - x^\dag\|: \|y^\d-y\|\le \d\right\} = o(\d^{1/2})
\end{align}
as $\d\to 0$, then $x^\dag = x^*$. 
\end{theorem}

\begin{proof}
Let $\a:= \a(\d, y^\d)$.  According to (\ref{8}), we have $\|x_\a^\d - x^\dag\| < \rho$ 
for all $y^\d$ satisfying $\|y^\d - y\|\le \d$ provided $\d>0$ is sufficiently small. Thus 
$x_\a^\d$ is an interior point of $\mbox{dom}(F)$ and consequently there holds the Euler equation 
$$
F'(x_\a^\d)^*(F(x_\a^\d)-y^\d) + \a (x_\a^\d - x^*) =0.
$$
Manipulating this equation we then obtain
\begin{align*}
(\a I + A^*A) (x_\a^\d - x^\dag) 
& = \a (x^* - x^\dag) + F'(x_\a^\d)^* (y^\d - F(x_\a^\d)) + A^* A (x_\a^\d - x^\dag) \\
& = \a (x^* - x^\dag) + (F'(x_\a^\d)^* - A^*) (y^\d - F(x_\a^\d)) \\
& \quad \, - A^* (F(x_\a^\d)- y^\d - A(x_\a^\d - x^\dag))
\end{align*}
which gives 
\begin{align*}
x_\a^\d - x^\dag
& = (\a I + A^*A)^{-1} \left(\a (x^* - x^\dag) + A^*(y^\d-y)\right) \\
& \quad \, + (\a I + A^*A)^{-1} (F'(x_\a^\d)^* - A^*) (y^\d - F(x_\a^\d)) \\
& \quad \, - (\a I + A^*A)^{-1} A^* (F(x_\a^\d)- y - A(x_\a^\d - x^\dag)).
\end{align*}
By using the definition of $\hat x_\a^\d$, we thus have  
\begin{align}\label{TRH:10.6.21}
x_\a^\d-\hat{x}_\a^\d
& = -(\a I+ A^*A)^{-1} A^* \left(F(x_\a^\d)-y-A(x_\a^\d-x^\dag)\right) \nonumber \\
& \quad \, + (\a I+ A^*A)^{-1} \left(F'(x_\a^\d)^*-A^*\right)\left(y^\d-F(x_\a^\d)\right).
\end{align}
Since $x_\a^\d \in B_\rho(x^\dag)$, we may use Assumption \ref{Ass.1} to obtain 
\begin{align*}
x_\a^\d - \hat x_\a^\d 
& = - \int_0^1 (\a I + A^*A)^{-1} A^* A
\left(R(x^\dag + t(x_\a^\d - x^\dag), x^\dag) - I\right) (x_\a^\d - x^\dag) dt \\
& \quad \, + (\a I + A^*A)^{-1} \left(I - R(x^\dag, x_\a^\d)^*\right) 
F'(x_\a^\d)^* (y^\d - F(x_\a^\d)).
\end{align*}
and consequently 
\begin{align*}
\|x_\a^\d - \hat x_\a^\d \| 
& \le  \int_0^1 t \kappa_0 \|x_\a^\d - x^\dag\|^2 dt
+ \frac{1}{\a} \|I - R(x^\dag, x_\a^\d)\| \|F'(x_\a^\d)^* (y^\d - F(x_\a^\d))\| \displaybreak[0]\\
& \le \frac{1}{2} \kappa_0 \|x_\a^\d - x^\dag\|^2 + \kappa_0 \|x_\a^\d - x^\dag\| \|x_\a^\d - x^*\| 
\le C_0 \|x_\a^\d - x^\dag\|,
\end{align*}
where $C_0:= (3\rho/2 + \|x^* - x^\dag\|) \kappa_0$. Therefore 
\begin{align*}
\|\hat x_\a^\d - x^\dag \| \le (1 + C_0) \|x_\a^\d - x^\dag\|.
\end{align*}
By virtue of (\ref{8}) we then have
\begin{align}\label{TRH:8.6.21}
\sup\left\{\|\hat x_\a^\d-x^\dag\|: \|y^\d-y\| \le \d\right\} = o(\d^{1/2}).
\end{align}

Based on (\ref{TRH:8.6.21}) and the assumption on $\sigma(AA^*)$, we will show that 
for $\d_k:= \la_k$ there exists $y^{\d_k}$ satisfying $\|y^{\d_k} - y\| \le \d_k$ 
such that 
$$
\d_k /\a_k \to 0 \quad \mbox{ as } k \to \infty,
$$
where $\a_k := \a(\d_k, y^{\d_k})$. To see this, let $\{F_\la\}$ denote the spectral 
family generated by the self-adjoint operator $AA^*$ and for each $k$ we set
$$
G_{\la_k}:=F_{3\la_k/2}-F_{\la_k/2}.
$$
Then $G_{\la_k}$ is an orthogonal projection on $Y$ and commutes with $(\a I+AA^*)^{-1}$. 
We then follow an idea from \cite{EG1988,N1997} and consider the perturbed data defined by
$$
y^{\d_k}:=y - \d_k G_{\la_k} z_k,
$$
where
$$
z_k:=\left\{\begin{array}{lll}
\|G_{\la_k}A (x^\dag - x^*)\|^{-1} A (x^\dag-x^*), & \mbox{ if }
G_{\la_k}A (x^\dag-x^*) \ne 0,\\
\mbox{arbitrary with } \|G_{\la_k}z_k\|=1, & \mbox{ otherwise}.
\end{array}\right.
$$
Since $\la_k \in \sigma(AA^*)$, such $y^{\d_k}$ is well-defined. It is clear that 
$\|y^{\d_k}-y\|=\d_k$. Let 
$$
\hat x_{\a_k} := x^\dag + \a_k (\a_k I + A^*A)^{-1} (x^*-x^\dag).
$$
From the definitions of $\hat x_{\a_k}^{\d_k}$ and $\hat x_{\a_k}$ 
it follows that
$$
\hat x_{\a_k}^{\d_k} - \hat x_{\a_k} = -\d_k (\a_k I+A^*A)^{-1} A^* G_{\la_k} z_k.
$$
Hence, by using $G_{\la_k}^2 = G_{\la_k}$, the definition of $z_k$ and $\|G_{\la_k} z_k\|=1$, we have
\begin{align*}
& 2\left\l \hat x_{\a_k}-x^\dag, \hat x_{\a_k}^{\d_k} - \hat x_{\a_k}\right\r \\
&=-2\a_k \d_k \left\l (\a_k I+AA^*)^{-1} A (x^*-x^\dag), (\a_k I+AA^*)^{-1} G_{\la_k} z_k\right\r \displaybreak[0]\\
&=2\a_k \d_k \left\l (\a_k I+AA^*)^{-1} G_{\la_k} A (x^\dag-x^*), (\a_k I+AA^*)^{-1} G_{\la_k} z_k\right\r \displaybreak[0]\\
&=2\a_k\d_k \|G_{\la_k} A (x^\dag-x^*)\| \|(\a_k I+AA^*)^{-1} G_{\la_k} z_k\|^2 \displaybreak[0]\\
&\ge 0
\end{align*}
and
\begin{align*}
\|\hat x_{\a_k}^{\d_k} - \hat x_{\a_k}\|^2
&=\d_k^2 \|(AA^*)^{1/2} (\a_k I+AA^*)^{-1} G_{\la_k} z_k\|^2 \displaybreak[0] \\
& =\d_k^2 \int_{\la_k/2}^{3\la_k/2} \frac{\la}{(\a_k+\la)^2} d\|F_\la G_{\la_k} z_k\|^2 \displaybreak[0]\\
& \ge \d_k^2 \frac{\la_k/2}{(\a_k+3\la_k/2)^2} \|G_{\la_k} z_k\|^2 \\
& = \d_k^2 \frac{\la_k/2}{(\a_k+3\la_k/2)^2}.
\end{align*}
Therefore, by using the identity
\begin{align*}
\|\hat x_{\a_k}^{\d_k} - x^\dag\|^2
=\|\hat x_{\a_k} - x^\dag\|^2 + \|\hat x_{\a_k}^{\d_k} - \hat x_{\a_k}\|^2
+ 2 \left\l \hat x_{\a_k} - x^\dag, \hat x_{\a_k}^{\d_k} - \hat x_{\a_k}\right\r,
\end{align*}
we obtain
\begin{align}\label{MDPS2}
\|\hat x_{\a_k}^{\d_k} - x^\dag\|^2 
\ge \|\hat x_{\a_k}^{\d_k} - \hat x_{\a_k}\|^2
\ge \d_k^2 \frac{\la_k/2}{(\a_k+3\la_k/2)^2} 
= \frac{\d_k^2/\la_k}{2(\a_k/\la_k+ 3/2)^2}.
\end{align}
Since $\d_k = \la_k$, it then follows from (\ref{TRH:8.6.21}) and (\ref{MDPS2}) that
$$
\frac{\d_k}{2(\a_k/\d_k+3/2)^2} \le \|\hat x_{\a_k}^{\d_k} - x^\dag\|^2 = o(\d_k)
$$
which implies that $\a_k/\d_k \to \infty$ as $k\to \infty$. 

Finally we show $x^\dag = x^*$. Let
$
C_1 := \sup_{x\in B_\rho(x^\dag)} \|F'(x)\|
$
which is finite by Assumption \ref{Ass.1}. By the definition of $\a_k$ and 
the Euler equation for $x_{\a_k}^{\d_k}$ we have
\begin{align*}
\|x_{\a_k}^{\d_k}-x^*\| 
= \frac{1}{\a_k} \|F'(x_{\a_k}^{\d_k})^* (F(x_{\a_k}^{\d_k})-y^{\d_k})\| 
\le \frac{C_1}{\a_k} \|F(x_{\a_k}^{\d_k})-y^{\d_k}\| 
\le \frac{C_1 \tau \d_k}{\a_k} \to 0
\end{align*}
as $k \to \infty$. Since (\ref{8}) also implies $\|x_{\a_k}^{\d_k}-x^\dag\| \to 0$ as $k\to \infty$,
we therefore have $x^\dag = x^*$. 
\end{proof}

The saturation result given in Theorem \ref{TRH.thm1} covers not only Rule \ref{Rule:DP} but also other 
possible parameter choice rules including the sequential discrepancy principle (\cite{HM2012}) below. 

\begin{Rule}[Sequential discrepancy principle]\label{SDP}
Let $\tau>1$, $\a_0>0$ and $0<\gamma <1$ be given numbers, and let 
$$
\Delta_\gamma:=\{\a_0 \gamma^j: j =0, 1, \cdots\}. 
$$
Define $\a(\d, y^\d)$ to be the largest number in $\Delta_\gamma$ such that 
$\|F(x_{\a(\d, y^\d)}^\d) - y^\d \| \le \tau \d$. 
\end{Rule}

We remark that for nonlinear ill-posed problems there are some issues concerning Rule \ref{Rule:DP} 
which require serious consideration. First, the existence of $\a(\d, y^\d)$ in  Rule \ref{Rule:DP} 
is not always guaranteed in general. Furthermore, even if 
such an $\a(\d, y^\d)$ exists, additional work needs to be executed to find it which can be 
time-consuming. The sequential discrepancy principle can resolve these issues automatically. 
Indeed, by the definition of $x_\a^\d$ it is easy to see that $\lim_{\a\to 0} \|F(x_\a^\d) - y^\d\|\le \d$
which implies the existence of $\a(\d, y^\d)$ by Rule \ref{SDP}. According to Theorem \ref{TRH.thm1},
under merely Assumption \ref{Ass.1}, Tikhonov regularization (\ref{TRH3}) with Rule 
\ref{SDP} exhibits the saturation. 

Next we will provide a saturation result under the following Lipschitz continuity 
of the Fr\'{e}chet derivative of the forward operator $F$. 

\begin{Assumption}\label{TRH:Lip}
{\it There is $\rho>0$ such that $B_\rho(x^\dag) \subset \emph{dom}(F)$ and $F$ is Fr\'{e}chet 
differentiable on $B_\rho(x^\dag)$. Moreover, there is $L\ge 0$ such that
$$
\|F'(x)-F'(z)\|\le L\|x-z\|, \qquad \forall x, z\in B_\rho(x^\dag).
$$
}
\end{Assumption}

We will show that, under Assumption \ref{TRH:Lip}, Tikhonov regularization (\ref{TRH3}) 
with the regularization parameter chosen by Rule \ref{Rule:DP} or Rule \ref{SDP} does not 
possess the convergence rate $o(\d^{1/2})$ in the worst case scenario even if 
$x^\dag - x^* \in \mbox{Ran}((A^*A)^\nu)$ with $\nu\ge 1/2$. We need the following 
result which is essentially contained in the proof of \cite[Theorem 2.4]{EKN89}, 
see also \cite{EHN96,TJ2003}.

\begin{lemma} \label{TRH:t22}
Let Assumption \ref{TRH:Lip} hold and let $x_\a^\d$ be a solution of (\ref{TRH3}) 
with $x_\a^\d \in B_\rho(x^\dag)$. If $x^*-x^\dag=F'(x^\dag)^* u$ for some $u\in Y$ 
with $L \|u\| < 1 $, then
\begin{equation} \label{TRH:f22}
\|x_\a^\d - x^\dag\| \le \frac{\d + \a
\|u\|}{\sqrt{\a(1-L\|u\|)}} \quad \mbox{and} \quad \|
F(x_\a^\d) - y^\d \| \le \d + 2 \a \|u\|.
\end{equation}
\end{lemma}

\begin{theorem}\label{SMD:thm2}
Let Assumption \ref{TRH:Lip} hold and assume, for the operator $A: = F'(x^\dag)$, there exists a 
sequence of positive numbers $\{\la_k\}\subset \sigma(AA^*)$ such that 
$\lim_{k\to \infty} \la_k =0$. Consider the Tikhonov regularization (\ref{TRH3}) 
and let $\a(\d, y^\d)$ be determined by either Rule \ref{Rule:DP} or Rule \ref{SDP} with $\tau>1$. 
If $x^\dag - x^* = F'(x^\dag)^* u$ for some $u \in \N(F'(x^\dag)^*)^\perp$ with $L\|u\|<1$ 
and if
\begin{align}\label{13}
\sup\left\{\|x_{\a(\d, y^\d)}^\d - x^\dag\|: \|y^\d - y\| \le \d\right\} = o(\d^{1/2})
\end{align}
as $\d \to 0$, then $x^\dag = x^*$.
\end{theorem}

\begin{proof}
Let $\a :=\a(\d, y^\d)$. From (\ref{13}) it follows that $x_\a^\d \in B_\rho(x^\dag)$ for 
small $\d>0$. Thus, by virtue of (\ref{TRH:10.6.21}) and Assumption \ref{TRH:Lip} we have
\begin{align*}
\|x_\a^\d-\hat{x}_\a^\d\|
& \le  \frac{1}{2\sqrt{\a}} \|F(x_\a^\d)-y-A(x_\a^\d-x^\dag)\| 
+ \frac{1}{\a} \|F'(x_\a^\d)-A\| \|y^\d-F(x_\a^\d)\| \displaybreak[0] \\
& \le \frac{1}{4\sqrt{\a}} L \|x_\a^\d-x^\dag\|^2
+ \frac{1}{\a} L\|x_\a^\d-x^\dag\| \|y^\d-F(x_\a^\d)\|.
\end{align*}
With the help of Lemma \ref{TRH:t22} we then have
\begin{align}\label{TRH.12}
\|x_\a^\d-\hat{x}_\a^\d\|
\le  \frac{L(\d + \a \|u\|)}{4\a \sqrt{1-L\|u\|}} \|x_\a^\d-x^\dag\|
+ \frac{1}{\a} L\left(\d + 2\a \|u\|\right) \|x_\a^\d-x^\dag\|.
\end{align}
We next claim that 
\begin{align}\label{TRH.13}
\a \ge \frac{(\tau -1)\gamma \d}{2\|u\|}. 
\end{align}
Indeed, if $\a$ is determined by Rule \ref{Rule:DP}, then by using the second 
inequality in Lemma \ref{TRH:t22} we have
$$
\tau \d = \|F(x_\a^\d)-y^\d\| \le \d + 2\a \|u\|
$$
which implies that $\d \le 2\a\|u\|/(\tau-1)$ and hence (\ref{TRH.13}) holds as $0< \gamma <1$. 
If $\a$ is determined by Rule \ref{SDP}, then by using again the second inequality in 
Lemma \ref{TRH:t22} we can obtain 
$$
\tau \d < \|F(x_{\a/\gamma}^\d) - y^\d\| \le \d + \frac{2\a}{\gamma} \|u\|
$$
which implies again (\ref{TRH.13}). Therefore, it follows from (\ref{TRH.12}) and (\ref{TRH.13}) that 
\begin{align*}
\|x_\a^\d-\hat{x}_\a^\d\| \le C_0 L\|u\| \|x_\a^\d-x^\dag\|,
\end{align*}
where
$$
C_0 := 2 + \frac{2}{(\tau-1)\gamma} + \frac{2 + (\tau-1)\gamma}{4(\tau-1) \gamma \sqrt{1-L\|u\|}}.
$$
Consequently
$$
\|\hat x_\a^\d -x^\dag \| \le (1+ C_0 L\|u\|) \|x_\a^\d - x^\dag\|.
$$
and thus we may use (\ref{13}) to conclude 
\begin{align*}
\sup\left\{\|\hat x_\a^\d-x^\dag\|: \|y^\d-y\| \le \d\right\} = o(\d^{1/2}).
\end{align*}
Now we can follow the same argument in the proof of Theorem \ref{TRH.thm1} to 
conclude $x^\dag = x^*$.  
\end{proof}

\end{document}